\documentclass[oneside]{amsart}
\usepackage{amsmath, amssymb}
\usepackage{amsthm}
\newtheorem{theorem}{Theorem}[section]
\newtheorem{proposition}[theorem]{Proposition}
\newtheorem{lemma}[theorem]{Lemma}
\newtheorem{corollary}[theorem]{Corollary}

\newtheorem{example}{Example}

\usepackage{color,bm}

\def\F{\mathcal{F} } 
\def\Et{\widetilde{E} } 
\def\Eh{\hat{E} } 
\def\Eb{\overline{E} } 
\def\Ft{\widetilde{\mathcal{F}} } 

\def\Ut{\widetilde{U} }

\def\R{\mathbb{R} } 
\def\Z{\mathbb{Z} } 
\def\nbd{neighborhood } 
\def\nbds{neighborhoods } 
\def\R{\mathbb{R} } 
\def\iff{if and only if }

\title[Recurrence, pointwise almost periodicity and orbit closure relation]
{Recurrence, pointwise almost periodicity  and orbit closure relation for flows and foliations}
\author{Tomoo Yokoyama}
\date{\today}

\address{Department of Mathematics, Hokkaido University, 
Kita 10, Nishi 8, Kita-Ku, Sapporo, Hokkaido, 060-0810, Japan \\
}
\email{yokoyama@math.sci.hokudai.ac.jp}

\thanks{The author is partially supported 
by the JST CREST Program at Department of Mathematics,  
Hokkaido University.}

\begin{document}

\maketitle

\begin{abstract}
In this paper, 
we obtain a characterization of  the recurrence of 
a continuous vector field $w$ of a closed connected surface $M$ 
as follows.  
The following are equivalent: 
1) 
$w$ is pointwise recurrent. 
2)
$w$ is pointwise almost periodic.
3) 
$w$ is minimal or pointwise periodic.
Moreover, 
if $w$ is regular, then 
the following are equivalent: 
1) 
$w$ is pointwise recurrent. 
2)
$w$ is minimal or 
the orbit space $M/w$ is either $[0,1]$, or $S^1$. 
3)  
$R$ is closed   
(where 
$R := \{ (x,y) \in M \times M \mid y \in \overline{O(x)} \}$  
is the orbit closure relation).  
On the other hand, 
we show that 
the following are equivalent 
for a codimension one foliation $\mathcal{F}$ on a compact connected manifold: 
1) $\mathcal{F}$ is pointwise almost periodic. 
2) $\mathcal{F}$ is minimal or compact. 
3) $\mathcal{F}$ is $R$-closed. 
Also we show that if 
a foliated space on a compact metrizable space is either 
minimal or
both compact and without infinite holonomy, 
then it is $R$-closed.
\end{abstract}

\section{Preliminaries}
In \cite{AGW} and \cite{Ha}, 
it is showed that  
the following properties are equivalent 
for a finitely generated group $G$ on 
either a compact zero-dimensional space or  
 a graph $X$: 
1) $(G,X)$ is pointwise recurrent. 
2) $(G,X)$ is pointwise almost periodic. 
3) The orbit closure relation $R= \{ (x,y) \in X \times X \mid y \in \overline{G(x)} \}$ is closed.

In this paper, 
we study the equivalence for these three notions for 
 vector fields on surfaces  
 and 
 codimension one foliations on manifolds, 
 and show the some equivalence. 
%
We assume that 
every quotient space has the usual quotient topology and 
that every decomposition consists of 
non empty elements.  
By a decomposition, 
we mean a family $\F$ of pairwise disjoint subsets of a set $X$ 
such that $X = \sqcup \mathcal{F}$. 
Let $X$ be a 
topological space 
and 
$\mathcal{F}$ a decomposition of $X$. 
For any $x \in X$, 
denote by $L_x$ the element of $\mathcal{F}$ containing $x$. 
%
Write $E_{\F} := \{ (x, y) \mid y \in L_x \}$. 
Then $E_{\F}$ is an equivalence relation 
(i.e. a reflexive, symmetric and transitive relation). 
For a (binary) relation $E$ on a set $X$ 
(i.e. a subset of $X \times X$), 
let $E(x) := \{ y \in X \mid (x,  y) \in  E\}$ 
for an element $x$ of $X$.
For any $A \subseteq X$, 
let $E(A) := \cup_{y \in A}E(y)$.  
$A$ is said to be $E$-saturated if 
$A = \cup_{x \in A} E(x)$. 
Let $1_X := \{(x, x) \mid x \in X\}$ be the diagonal on $X \times X$. 
Thus $1_X \subseteq E$ if and only if 
$E$ is reflexive 
(i.e. $x \in E(x)$ for all $x \in X$).  
Let $E^{-1} := \{ (y, x) \mid (x, y) \in  E\}$ 
(i.e. the image of $E$ under the bijection $T: X \to X$ 
which interchanges coordinates). 
Clearly, 
$E$ is symmetric if and only if  
$E = E^{-1}$. 
For any relation $E$ on $X$, 
$E$ is transitive if and only if 
$E(E(x)) \subseteq E(x)$. 
%
For an equivalence relation $E$, 
the collection of equivalence classes 
$\{E(x) \mid x \in X\}$ is a decomposition of X, 
denoted by $\F_E$. 
Note that 
decompositions (consisting of  nonempty elements) can be corresponded to equivalence relations. 
Therefore we can identify decompositions with equivalence relations. 
For a relation $E$ on a topological space $X$, 
$E$ define the relation $\Eh$ on 
$X$ with $\Eh (x) = \overline{E(x)}$. 
Denote by $\Eb$ the closure of $E$ in $X \times X$. 
We call 
$E$ 
pointwise almost periodic if $\Eh$ 
is an equivalence relation,  
$R$-closed if 
$\Eh$ is closed (i.e. $\Eh = \Eb$), 
compact
if 
each element of $E$ is compact, 
and 
it is minimal if 
each element of $E$ is dense in $X$.  
By identification, 
we also said that 
$\F$ is $R$-closed if so is $E_{\F}$ and others are defined in similar ways.  
Notice that 
if $\mathcal{F}$ is either a foliation or the set of orbits of a flow, 
then $\Eh_{\mathcal{F}}$ is transitive. 
By a flow, we mean a continuous action of a topological group $G$ on $X$. 
We call that 
$\mathcal{F}$ is trivial if it consists of singletons or 
is minimal. 
We characterize the transitivity for $\Eh$. 

\begin{lemma}\label{lem:0085}
$\Eh$ is transitive if and only if 
$E( \Eh (x)) \subseteq \Eh (x)$.  
%
When $\Eh$ is an equivalence relation,  
this says that each $\Eh (x)$ is a union of $E$-equivalence classes.
\end{lemma}

\begin{proof}
Since $\Eh (x)$ is closed, 
$E(y) \subseteq \Eh (x)$ implies $\Eh (y) = \overline{E (y)} \subseteq \Eh (x)$.
\end{proof}

Now we state a useful tool.

\begin{lemma}\label{lem:009}
If $E$ is an equivalence relation, 
then $\Eh$ is an equivalence relation 
\iff it is symmetric 
(i.e. symmetry implies transitivity).
\end{lemma}

\begin{proof}
Let $y \in \Eh (x)$ and $z \in E(y)$. 
By the symmetry assumption, 
we have $x \in \Eh (y)$. 
Since $E$ is an equivalence relation,  
$E(y) = E(z)$ and so
$\Eh (y) = \Eh (z)$. 
So $x \in \Eh (z)$ 
and, by symmetry again, 
$z \in \Eh (x)$ and so 
$E(y) \subseteq \Eh (x)$. 
Then $E(\Eh(x)) \subseteq \Eh (x)$. 
Hence $\Eh$ is transitive by Lemma \ref{lem:0085}.
\end{proof}

Notice that the twist map $T: X \to X$ 
is a homeomorphism 
and
so $E = E^{-1}$ implies 
$\Eb = \Eb^{-1}$. 
In particular, 
$\Eb$ is symmetric whenever
$E$ is an equivalence relation. 
In particular, 
the previous lemma implies the following statement. 

\begin{corollary}\label{cor:15}
Suppose that  
$E$ is an equivalence relation. 
If $\Eh$ is closed, 
then $\Eh$ is an equivalence relation.
 \end{corollary}

This is interpreted as the following statement.    
 
\begin{corollary}\label{cor:16}
If ${\F}$ is an $R$-closed decomposition, 
then ${\F}$ is pointwise almost periodic. 
\end{corollary}

The converse of this corollary is not true 
 (see Example \ref{ex:04}). 
A map $p: X \to Y$ is said to be perfect if 
it is continuous, closed, surjective and 
each fiber $p^{-1}(y)$ for any $y \in Y$ is compact. 
Recall that 
a net is a function from a directed set to a topological space.  

\begin{lemma}\label{lem:0084}
If a relation $E$ is closed 
and 
$X$ is $T_1$, 
then $E(x)$ is closed for any $x \in X$. 
Moreover if 
$E$ is an equivalence relation and 
$X$ is compact Hausdorff, 
then 
the quotient map $q:X \to X/E$ is perfect. 
\end{lemma}

\begin{proof}
Note that 
each singleton is closed in a $T_1$-space. 
Since $\{x \} \times E(x) = E \cap (\{x \} \times X)$, 
we have $E(x)$ is closed in $X$.  
Suppose that 
$E$ is an equivalence relation and 
$X$ is compact Hausdorff. 
We will show that 
$q$ is closed. 
Otherwise 
there is a closed subset $B$ of $X$ such that 
$E(B)$ is not closed. 
Fix any $y \in \overline{E(B)} - E(B)$. 
Let 
$(y_{\alpha})$ be a net in $B$ and 
$x_{\alpha} \in E(y_{\alpha})$ such that 
$y_{\alpha} \to y$. 
Then $(x_{\alpha}, y_{\alpha}) \in E$. 
Since $B$ is closed and so compact, 
we may assume that 
$(x_{\alpha})$ converges to some element $x \in B$, by taking a subnet of $(x_{\alpha})$. 
Since $y \notin  E(x) \subseteq E(B)$, 
we have $(x, y) \notin E$. 
Since $(x_{\alpha}, y_{\alpha}) \in E$ and 
they converge to $(x, y)$,   
we have $(x, y) \in {\Eb}$, 
which contradicts that 
$E$ is closed. 
Since 
each fiber of every element of $X/E$ is of the form $E(x)$ for some $x \in X$  
and so compact, 
we have that 
$q$ is perfect.  
\end{proof}

A pointwise almost periodic $\F$ is 
weakly almost periodic in the sense of Gottschalk \cite{G} 
if 
the saturation 
$\cup_{x\in A} \overline{L_x}$ 
of orbit closures for any closed subset $A$ of $X$ is 
closed. 
Note that 
a pointwise almost periodic decomposition 
$\F$ is weakly almost equivalent 
if and only if  
the quotient map $q: X \to X/\Eh_{\F} 
(= X/\hat{\F})
$ 
is closed. 
This implies the following interpretation.

\begin{lemma}\label{lemma1}
Suppose that  $X$ is  a compact Hausdorff space. 
If a decomposition $\F$ is $R$-closed, then 
$\F$ is weakly almost periodic. 
\end{lemma}

\section{General cases for Flows and Foliated spaces}

For an equivalence relation $E$ on $X$ 
and for an element $x$ of $X$,  
recall that 
the class $\Et (x)$ of $x$ 
is defined 
by $\Et (x) := \{ y \in X \mid \Eh(y) = \Eh(x) \}$ \cite{HS}. 
These classes imply an equivalence relation on $E$. 
Then the quotient space by this equivalence relation is denoted by $X/\widetilde{E}$ 
and is called the orbit class space {\cite{BHSV} (or the quasi-orbits space)}. 
Write $\widetilde{\F} := \F_{\Et}$. 
By identification, 
$\widetilde{L} \in \Ft$ if and only if 
there is $L \in \F$ such that 
$\widetilde{L} = \{ y \in X \mid \overline{L} = \overline{L_y} \}$. 
%
%
In the case that 
$\mathcal{F}$ is pointwise almost periodic, 
$X/{\widetilde{\F}}$ is exactly  
the quotient space of  $\F_{\Eh}$ with 
the quotient topology and 
$X/\widetilde{\F} = X/\Et_{\F} = X/\Eh_{\F}$. 
If $\Eh$ is an equivalence relation, 
then $\Et = \Eh$ and so 
$X/\Et = X/\Eh$. 
Recall the following fact. 

\begin{lemma}\label{lem:0031}
(\cite{Bo} Proposition 8.3.8, 8.6.14)
Let $E$ be an equivalence relation in  $X$.  
If $X/E$ is Hausdorff, then $E$ is closed. 
If $X$ is $T_3$,  
then the converse holds.  
\end{lemma}

Recall that  
a topological space is $T_3$ if 
it is Hausdorff and regular. 
This fact implies 
%
%
%
a useful tool. 

\begin{lemma}\label{lem:003}
Let $\mathcal{F}$ be 
a pointwise almost periodic decomposition 
on a $T_3$ space $X$. 
Then 
$\mathcal{F}$ is  $R$-closed  
if and only if 
$X/\widetilde{\F}$ is Hausdorff. 
\end{lemma}

\begin{proof}
Since $\mathcal{F}$ is pointwise almost periodic, 
we have that 
$\Eh_{\F}$ is an equivalence relation. 
Suppose that 
$\mathcal{F}$ is  $R$-closed. 
Then $\Eh_{\F}$ is closed. 
By Lemma \ref{lem:0031}, 
$X/\Eh_{\F} = X/\widetilde{\F}$ is Hausdorff. 
Conversely, 
suppose that 
$X/\widetilde{\F}$ is Hausdorff. 
Since $X/\Eh_{\F} = X/\widetilde{\F}$, 
Lemma \ref{lem:0031} implies that 
$\Eh_{\F}$ is closed. 
This shows that $\F$ is $R$-closed. 
\end{proof}

In this lemma, 
the 
regularity of $X$ 
is necessary 
even if $X$ is Hausdorff (see Example \ref{ex:0021}). 
However we don't know whether 
the regularity 
is 
necessary when $\F$ is 
the orbit space of a flow.  
Now we state an observation.

\begin{lemma}\label{lem:011}
Let $\mathcal{F}$ be a decomposition on $X$. 
If 
either 
$\mathcal{F}$ is minimal
or
$X/\mathcal{F}$ is Hausdorff, 
then 
$\mathcal{F}$ is  $R$-closed. 
\end{lemma}

\begin{proof}
If $\mathcal{F}$ is minimal, 
then $R =\Eh_{\mathcal{F}}$ is a singleton and so 
closed. 
If $X/\mathcal{F}$ is Hausdorff, 
then all orbits are closed and so 
$\mathcal{F}$ is pointwise almost periodic. 
By Lemma \ref{lem:003}, 
if suffices to show that 
$X/\widetilde{\F}$ is Hausdorff.  
But 
$X/\mathcal{F} = X/\widetilde{\F}$ is Hausdorff.
\end{proof}

This observation implies the following statement.

\begin{proposition}\label{cor:01}
A foliated space $(X, \F)$
on a compact metrizable space 
either 
which is 
minimal or which is compact and without infinite holonomy,  
is $R$-closed. 
\end{proposition}

\begin{proof}
If 
$\F$  
is minimal, 
then Lemma \ref{lem:011} implies that 
$\F$ is $R$-closed. 
If $\F$ is a compact foliated space without infinite holonomy, 
then Theorem 4.2 \cite{Ep} implies that 
$X/\F$ is Hausdorff. 
By Lemma \ref{lem:011}, 
we have that 
$\F$ is $R$-closed. 
\end{proof}

Note that 
Epstein \cite{Ep2} et al have shown that 
each compact codimension two foliation on a compact manifold
has finite holonomy. 
This implies that 
the set of $R$-closed codimension two foliations 
contains properly the set of codimension two 
foliations which are minimal or compact. 
Therefore 
the author is interested in 
the characterization of  the $R$-closedness for 
codimension two foliations. 
Our statement is similar to the statement \cite{V}(A.2)3. 
However 
notice that 
it is not true, 
because pointwise almost periodicity does not 
correspond to $R$-closedness (cf. Example \ref{ex:04}).

\section{On the $R$-closedness of $\mathcal{F}$}
In this section, 
we will show that the following three notions are equivalent  
for an equivalence relation $E$ on 
a compact Hausdorff space
: 
%
1) $R$-closed, 
2) $D$-stable, and 
3) $L$-stable. 
%
%
%
%
A point $x$ in $X$ 
is said to be $D$-stable (or of characteristic 0) 
if $\Eh (x) =D(x)$ for any $x \in X$, 
where $D(x)$ is its (bilateral) prolongation defined as follows: 
$D(x) = 
\{ y \in X \mid y_{\alpha} \in  E(x_{\alpha}), y_{\alpha} \to y,  \text{ and } x_{\alpha} \to x 
\text{ for some nets } (y_{\alpha}),  (x_{\alpha}) \subseteq X \}$. 
An  
equivalence relation $E$
is said to be $D$-stable (or of characteristic 0) 
if each point is $D$-stable (i.e. $D = \Eh$). 
Note that 
some authors require also non-triviality for the 
definition of $D$-stability. 
%
Recall a well-known fact that 
an equivalence relation 
$E$ on a Hausdorff space is $D$-stable 
if and only if 
$\Eh$ is closed. 
This fact implies the following corollary. 
%
%
%
%
%

 \begin{corollary}\label{cor:00}
Suppose that 
$X$ is Hausdorff 
and 
that  
 $\mathcal{F}$ is a decomposition on $X$. 
Then 
$\mathcal{F}$ is $D$-stable 
if and only if 
$\mathcal{F}$ is 
$R$-closed. 
\end{corollary}

Now 
we consider 
the following Lyapunov stable type condition. 
%
We call that 
$E$ is $L$-stable 
if for any open neighborhood  $U $ of $\Eh (x)$ 
and for an element $x$ of $X$, 
there is a $E$-saturated open neighborhood  $V$ of $\Eh (x)$ 
contained in $U$. 
Note that 
this notion is similar to upper semicontinuity.

\begin{lemma}\label{lemma:001}
Assume that $E$ is an $L$-stable equivalence relation on $X$. 
If $X$ is $T_1$,  
then $\Eh$ is an equivalence relation. 
If $X$ is 
compact Hausdorff, 
then $\Eh$ is a closed equivalence relation. 
\end{lemma}

\begin{proof}
Fix $x \in \Eh(y)$.  
Then every neighborhood of $x$ meets $E(y)$. 
Since $E$ is $L$-stable, 
every neighborhood of $\Eh (x)$ contains $E(y)$. 
Because $X$ is $T_1$,  
the intersection of the neighborhoods of a closed set
is exactly the closed set itself. Hence, $y \in E(y) \subseteq \Eh(x)$.
This shows that $\Eh$ is symmetric. 
By Lemma \ref{lem:009},  
we have that 
$\Eh$ is an equivalence relation. 
%
Suppose that 
$X$ is $T_4$. 
To apply Theorem 3.10 \cite{Ke} for $\Eh$, 
we show that 
$\Ut := \cup \{ x \in U \mid \Eh (x) \subseteq U \}$  
is open for any open subset $U$ of $X$. 
Fix $ x \in \Ut$. 
By the normality,  
there is a closed neighborhood $V$ of $\Eh (x)$ 
contained in $U$.
Since $E$ is $L$-stable, 
there is an $E$-saturated neighborhood $A$ of $\Eh (x)$ 
contained in $V$. 
Since $V$ is closed, 
we have 
$\Eh (A) \subseteq V \subseteq U$ and so 
$\Eh (A) \subseteq \Ut$. 
Since $\Eh (A)$ is a $\Eh$-saturated \nbd of $x$, 
we have that 
$\Ut$ is open. 
By Theorem 3.10 \cite{Ke}, 
the quotient map $q: X \to X/\Eh$ is closed. 
Since each fiber $\Eh (x)$ for an element $x$ of $X$ is 
compact, 
we have that 
$q$ is perfect.  
The fact Theorem XI. 5.2 (p.235) \cite{Du} implies that  
%
$X/\Eh$ is Hausdorff. 
By Lemma \ref{lem:0031}, 
we have that  $\Eh$ is closed. 
\end{proof}

We state that 
the equivalence between 
the ($R$-)closedness and 
the $L$-stability. 
 
\begin{proposition}\label{proposition:001}
Suppose that 
$E$ is an equivalence relation on a compact Hausdorff space $X$. 
Then $\Eh$ is a closed equivalence relation 
if and only if 
$E$ is $L$-stable. 
\end{proposition}

\begin{proof}
By the previous lemma, 
it suffices to show that 
if $\Eh$ is a closed equivalence relation,  
then $E$ is $L$-stable. 
By Lemma \ref{lem:0084}, 
we have that 
the quotient map $q: X \to X/\Eh$ is closed. 
%
For any open neighborhood  $U$ of $\Eh (x)$ for each $x \in X$, 
we have 
$X -U$ is closed and so 
$q(X -U)$ is closed.  
Then 
$S := q^{-1}(q(X -U)) = \Eh (X - U)$ is closed such that 
$S \cap \Eh (x) = \emptyset$, 
because $q(x) \notin q(X - U)$. 
Then $X - S$ is open ${\Eh}$-saturated 
and so 
$E$-saturated such that 
$\Eh (x) \subseteq X - S \subseteq U$
\end{proof}

This implies the following result.  

 \begin{corollary}\label{cor:00}
Suppose that 
$\mathcal{F}$ is a decomposition on a compact Hausdorff space. 
Then 
$\mathcal{F}$ is $R$-closed  
if and only if 
$\mathcal{F}$ is $L$-stable. 
\end{corollary} 

We summarize the properties. 

 \begin{corollary}\label{cor:00}
Let $E$ be an equivalence relation on 
a compact Hausdorff space.  
The following are equivalent: 
\\
1) $E$ is $R$-closed (i.e. $\Eh$ is closed). 
\\
2) $E$ is $D$-stable. 
\\
3) $E$ is $L$-stable. 
\end{corollary}

\section{On the density of $E$}

Consider the case with dense elements. 
%
%
%

\begin{proposition}\label{prop:006}
Suppose that 
an equivalence relation $E$ of $X$ 
has a dense element.  
The following are equivalent: 
\\
1) 
$\Eh$ is symmetric. 
\\
2)
$\Eh$ is an equivalence relation. 
\\
3) 
$\Eh$ is a closed equivalence relation.
\\
4) 
$E$ is minimal. 
\end{proposition}

\begin{proof}
Trivially 3) $\Rightarrow$  2). 
By Lemma \ref{lem:009},
2) $\Longleftrightarrow$ 1). 
We show that 
2) $\Rightarrow$ 4). 
If $\Eh$ is an equivalence relation, 
then 
$X$ is decomposed into the closures of elements of $\mathcal{F}_E$. 
Since there is a dense element of $E$, 
we have $\F_{\Eh} = \{X \}$ and so 
$E$ is minimal. 
Finally, 
we show that 
4) $\Rightarrow$ 3). 
If $E$ is minimal, 
then the closure of 
each element of $\F_E$ is the entire $X$ and so  
$\Eh  = X \times X$ is a closed equivalence relation trivially. 
\end{proof}

Note that 
there is a recurrent flow with a dense orbit on 
a compact metrizable space 
such that 
$\Eh_{\F}$ is not symmetric 
but transitive,   
where $\F$ is the set of orbits of a flow (e.g. p.764 \cite{Go}).
We say that 
an element $x$ of $X$ is almost periodic if 
$\Eh (y) = \Eh (x)$ for any $y \in \Eh (x)$.  
Considering the closure of an orbit of each $x \in X$ as the whole topological space, 
we obtain the following statement. 

\begin{corollary}\label{cor:001}
Let 
$x$ be an element of $X$. 
The following are equivalent: 
\\
1) 
$x$ is almost periodic. 
\\
2)
the restriction $\Eh|_{\Eh(x)}$ to 
${\Eh(x)}$ is symmetric 
\\
3)
$\Eh|_{\Eh(x)}$ is a closed equivalence relation. 
\end{corollary}

\section{Codimension one foliations}

In this section, we consider $\mathcal{F}$ 
as a codimension one foliation. 
Let $M$ be a compact connected manifold and 
$\mathcal{F}$ a continuous codimension one foliation on $M$ 
tangent or transverse to the boundaries. 
Recall that 
$\mathcal{F}$ is said to be  compact if 
each leaf of $\mathcal{F}$ is compact. 
Note that 
every codimension one compact foliation on a compact manifold 
has no infinite holonomy.

\begin{lemma}\label{lem:010}
If $\mathcal{F}$ is  pointwise almost periodic, 
then $\mathcal{F}$ is minimal or compact. 
\end{lemma}


\begin{proof}
Recall each minimal set is either 
a closed leaf, an exceptional minimal set, or 
the whole manifold. 
Since $\mathcal{F}$ is pointwise almost periodic, 
any compact leaf of $\mathcal{F}$ has no infinite holonomy. 
Suppose 
$\F$ is not minimal. 
By the almost periodicity, 
each proper leaf is compact 
and 
there are no locally dense leaves.  
Then $\mathcal{F}$ consists of 
compact leaves and 
exceptional leaves. 
We show that 
there is a compact leaf $L$ of $\mathcal{F}$. 
Otherwise, 
$M$ is the union of exceptional leaves. 
By Theorem \cite{S}, 
$M$ consists of finitely many exceptional minimal sets.  
Since each exceptional minimal set is nowhere dense, 
we have that $M$ is nowhere dense. 
This contradicts to that $M$ is a manifold.  
Then 
the union $C$ of compact leaves is nonempty.  
Since any compact leaf of $\mathcal{F}$ has no infinite holonomy, 
we have that  
$C$ is open. 
By Theorem 4.1.1(p.94)\cite{HH}, 
we have that $\partial C$ contains no exceptional minimal sets 
and so
$C$ is clopen.  
Hence $M$ consists of compact leaves. 
\end{proof}

\begin{theorem}\label{th:010}
The following are equivalent: 
\\
1) 
$\mathcal{F}$ is  pointwise almost periodic. 
\\
2)
$\mathcal{F}$ is $R$-closed. 
\\
3) 
$\mathcal{F}$ is minimal or compact. 
\end{theorem}

\begin{proof}
Since $M$ is compact Hausdorff, 
Corollary \ref{cor:16} implies that 
$2) \Rightarrow 1)$. 
By Lemma \ref{lem:010}, 
we obtain 
$1) \Rightarrow 3)$. 
Since any compact codimension one foliation 
$\F$ has no infinite holonomy, 
by Proposition \ref{cor:01}, we have 
$3) \Rightarrow 2)$. 
\end{proof}

\begin{corollary}\label{cor:010}
Suppose that  
$\mathcal{F}$ is not minimal. 
The following are equivalent: 
\\
1) 
$\mathcal{F}$ is  pointwise almost periodic. 
\\
2)
$\mathcal{F}$ is compact. 
\\
3) 
$M/\mathcal{F}$ is either closed interval or a circle. 
\\
4) 
$\mathcal{F}$ is $R$-closed. 
\end{corollary}

\begin{proof}
By Theorem \ref{th:010}, 
we obtain 
1) $\Longleftrightarrow$ 2) $\Longleftrightarrow$ 4).  
Taking the doubling of $M$, 
we may assume that 
$M$ is closed and 
$\mathcal{F}$ is transversally orientable. 
Lemma \ref{lem:011} implies 
that 
3) $\Rightarrow$ 4).
Thus it suffices to show that 
2) $\Rightarrow$ 3). 
Suppose that 
2) (and 4)) holds. 
By Lemma \ref{lem:003}, 
$M/\mathcal{F}$ is Hausdorff.  
Since $\mathcal{F}$ is compact, 
codimension one,  and transversally orientable, 
we have that $\F$ is without holonomy and so 
each leaf of $\mathcal{F}$ has a product neighborhood  of it. 
Hence
$M/\mathcal{F}$ is a closed $1$-manifold and so a circle.  
\end{proof}

\section{Flows on compact surfaces}

Let $X$ be a compact topological space and 
$G$ a topological group. 
Recall that a subset $S$ of $G$
is is said to be (left) syndetic if there is a compact set $K$ of $G$ with $KS = G$. 
Consider a flow $(X, G)$ (i.e. a continuous (left) action of $G$ on $X$). 
For a point $x \in  X$ 
and an open $U$ of $X$,  
let $N(x, U) = \{ g \in G \mid  gx \in  U\}$.  
We say that 
$x$ is an almost periodic point if 
$N(x, U)$ is syndetic for every neighborhood $U$ of $x$. 
A flow $G$ is pointwise almost periodic if 
every point $x \in X$ is almost periodic. 
Note that a flow on a compact Hausdorff space 
is pointwise almost periodic  if and only if 
the set of orbits is pointwise almost periodic.
A point $x \in X$ is recurrent (or Poisson stable) 
if 
$x \in \alpha(x) \cap \omega(x)$, 
where $\alpha(x)$ (resp. $\omega(x)$) is an alpha (resp. omega) limit set of $x$. 
A flow $G$ is (pointwise) recurrent 
if 
every point  of $X$ is recurrent. 
Note that a pointwise almost periodic flow 
on a compact Hausdorff space $X$ 
is pointwise recurrent 
and that 
a pointwise almost periodic flow on it is 
equivalent to a flow whose orbit closures form a decomposition of it.
$G$ is $R$-closed if 
$R := \{ (x, y) \mid y \in \overline{G(x)} \}$ is closed. 
Notice that 
an $R$-closed flow on a compact Hausdorff space $X$ 
is pointwise almost periodic.

Let $M$ be a compact  connected surface and 
$w$ a continuous vector field of  $M$. 
Note that 
if $w$ is pointwise recurrent, 
then 
$w$ is tangent to the boundary $\partial M$. 
In this section, 
we consider $\mathcal{F}$ as the set of the orbits of the vector field $w$. 
Write 
$M/w := M/\mathcal{F}$ and 
$M/\overline{w} := M/\widetilde{\F}$.  
Now $R = \Eh_{\F} = \{ (x, y) \mid y \in \overline{O_w(x)} \}$. 
%
%

\begin{lemma}\label{lem:00aa}
If $w$ is 
not minimal but 
pointwise recurrent,  
then $M = \mathrm{Sing}(w) \sqcup \mathrm{Per}(w)$.  
\end{lemma}

\begin{proof}
Since $w$ is recurrent, 
we have that 
$w$ preserves the boundary of $M$, 
has no saddles, 
and 
is non-wandering
(i.e. 
$x \in J^+(x)$ for any $x \in M$ where 
$J^+(x) = \{ y \in M \mid t_n x_n \to y \text{ for some } x_n \to x 
\text{ and } t_n \to + \infty \}$
). 
By taking the double of $M$ if necessary, 
we may assume that 
$M$ is closed. 
Since the Denjoy flow has wandering points, 
by Corollary 3.5 \cite{A}, 
we have $M = \mathrm{Sing}(w) \sqcup \mathrm{Per}(w)$.  
\end{proof}

This implies the following corollary.

\begin{corollary}\label{thm:00b} 
Let $w$ be a continuous vector field  of a compact connected surface $M$.  
The following are equivalent:
\\
1) 
$w$ is pointwise recurrent. 
\\
2) 
$w$ is pointwise almost periodic. 
\\
3)
$w$ is either minimal or pointwise periodic. 
\end{corollary}

Recall that 
a vector field is said to be nontrivial if 
it is neither identical nor minimal.

\begin{corollary}
Suppose that 
$w$ is nontrivial.  
Then 
$w$ is pointwise recurrent  
if and only if  
$w$ is pointwise periodic. 
\end{corollary}

Next, we show the openness of $\mathrm{Per}(w)$. 

\begin{lemma}\label{lem:00a}
Suppose that $w$ is pointwise recurrent. 
Then 
$\mathrm{Per}(w)$ is open. 
Moreover, 
if $\mathrm{Per}(w)$ is nonempty, 
then 
each connected component of $\mathrm{Per}(w)$ 
is an annulus or a torus, and 
$\mathrm{Per}(w)/w|_{\mathrm{Per}(w)}$ is a  
$1$-manifold. 
\end{lemma}

\begin{proof}
By Lemma \ref{lem:00aa},
we have $M = \mathrm{Sing}(w) \sqcup \mathrm{Per}(w)$.  
Suppose that there are periodic points. 
By the Flow box theorem, 
each periodic orbit has a product neighborhood  of it 
which consists of periodic orbits. 
Hence 
the set ${\mathrm{Per}(w)}$ of periodic orbits of $w$ is open 
and 
the quotient space of ${\mathrm{Per}(w)}$ is a $1$-manifold.
Then 
each connected component of $\mathrm{Per}(w)$ 
is an annulus or a torus. 
\end{proof}

From now on, 
we assume that 
$M$ is orientable. 
Note that 
we can obtain the similar results for the non-orientable case, 
by taking the double cover of $X$. 
For simplicity, we consider only 
the orientable case. 
Recall that 
$w$ is $R$-closed if 
$R := \{ (x,y) \in M \times M \mid y \in \overline{O(x)} \}$ 
is closed.

\begin{lemma}\label{lemma01}
If $w$ is nontrivial and 
$R$-closed, 
then 
$M/w$ is a closed interval or a circle, 
 and 
each connected component of 
$\partial {\mathrm{Per}(w)} := \overline{{\mathrm{Per}(w)}} - {\mathrm{Per}(w)}$  
is an element of $\mathcal{F}$ which is a center.
\end{lemma}

\begin{proof}
By Lemma \ref{lem:003}, 
we have that 
$M/\overline{w} = M/w$ is Hausdorff. 
By Lemma \ref{lem:00a}, 
we have that 
$M/w$ is a closed interval or a circle 
and  that 
each connected component of  
$\partial {\mathrm{Per}(w)}$ 
is an element of $\mathcal{F}$ which is a center. 
\end{proof}

Now, we state the characterization of $R$-closedness. 

\begin{theorem}\label{th:007} 
Suppose that $w$ is nontrivial. 
The following are equivalent: 
\\
1) 
$w$ is $R$-closed. 
\\
2)
$w$ consists of at most two centers and other periodic orbits. 
\\
3) 
The orbit space $M/w$ 
is either 
a circle reduced from a torus
or 
is a closed interval 
reduced from a closed disk, a closed annulus or a sphere. 

In particular, 
$w$ is pointwise periodic if one of equivalent conditions holds. 
\end{theorem}
 
\begin{proof}
By Lemma \ref{lemma01}, 
we have that 
1) $\Rightarrow$ 2) 
and that  
1) $\Rightarrow$ 3). 
Suppose that 
3) holds. 
Since $M/w$ is Hausdorff, 
we obtain that 
$M/w = M/\overline{w}$ is Hausdorff. 
By Lemma \ref{lem:003}, 
we have that $w$ is $R$-closed. 
Suppose that 
2) holds. 
By the Flow box theorem, 
we have that 
$\mathrm{Per}(w)$ is open and 
$\mathrm{Per}(w)/w|_{\mathrm{Per}(w)} 
= \mathrm{Per}(w)/\overline{w}|_{\mathrm{Per}(w)}$ is a $1$-manifold. 
Since singularities are finite and so isolated, 
$M/w = M/\overline{w}$ is Hausdorff. 
By Lemma \ref{lem:003}, 
we have that $w$ is $R$-closed. 
\end{proof}

Recall that 
$R$-closedness and 
$D$-stability for vector fields
are equivalent. 
Therefore Theorem 4.3, 4.4 \cite{APS} are 
related to our characterization.  
Consider the regular case. 
Recall that 
we say that 
a vector field $w$ is regular if 
$w$ is topologically equivalent to 
a vector field whose singular points are non-degenerate.  
In the two dimensional case, 
each singularity of 
a regular vector field is either 
a sink, a source, 
a (topological) saddle, or a center.

\begin{lemma}\label{lem:004} 
Suppose $w$ is nontrivial. 
Then 
$R$ is closed 
if and only if  
$w$ is 
pointwise almost periodic and 
topologically equivalent to 
a regular vector field. 
\end{lemma}

\begin{proof}
Suppose that 
$R$ is closed. 
By Theorem \ref{th:007}, 
we have that 
$w$ is pointwise periodic 
and 
is 
topologically equivalent to 
a regular vector field. 
%
%
%
%
Conversely, 
suppose that 
$w$ is pointwise almost periodic and regular. 
By Corollary \ref{thm:00b}, 
$w$ is pointwise periodic. 
Since $w$ has no saddle points, 
by the Poincar\'e-Hopf index formula, 
we obtain that 
$w$ consists of periodic orbits and at most two centers. 
By Theorem \ref{th:007},  we obtain 
$w$ is $R$-closed. 
\end{proof}

Now we state 
the following characterization 
in the regular case.

\begin{proposition}\label{prop:0055}
Let $w$ be a continuous regular nontrivial vector field of $M$. 
The following are equivalent: 
\\
1) 
$w$ is pointwise recurrent. 
\\
2)
$w$ is pointwise almost periodic. 
\\
3) 
$w$ is $R$-closed. 
\\
4)
$w$ is pointwise periodic. 
\\
5) $\mathrm{Sing}(w)$ consists of centers 
and $w$ has neither exceptional minimal sets nor limit cycles. 
\\
6) 
the orbit space $M/w$ is either 
a circle reduced from a torus, or 
a closed interval reduced from 
either 
a closed annulus, 
a closed disk, or a sphere. 
\end{proposition}

\begin{proof}
By Corollary \ref{thm:00b}  
and 
Lemma \ref{lem:004},  
1) $\Leftrightarrow$ 2) 
$\Leftrightarrow$ 3) 
$\Leftrightarrow$ 4). 
By Theorem \ref{th:007}, 
3) $\Leftrightarrow$ 6).  
Suppose that 
4) holds. 
Since $w$ is regular,  
$\mathrm{Sing}(w)$ consists of centers and 
so 5) holds. 
Finally, we show that 
5) $\Rightarrow$ 4). 
Suppose that 5) holds. 
We show that 
there are no exceptional orbits. 
Otherwise there is a local exceptional minimal set $K$.  
Then there is a periodic orbit which is contained in the closure of $K$, 
because $\mathrm{Sing}(w)$ consists of centers and 
$w$ has no exceptional minimal sets. 
So this means that 
there is a limit cycle, which contradicts. 
We show that 
each orbit is closed. 
Otherwise  
the boundary of the set of 
non-closed orbits contains minimal sets. 
Since there are no exceptional orbits and 
since 
$\mathrm{Sing}(w)$ consists of centers, 
the minimal sets are closed orbits 
and so 
there is limits cycles, which contradicts. 
Therefore 4) holds.  
\end{proof}

By the existence of the Denjoy flow,   
the condition that 
$w$ has no exceptional minimal sets in 4), is necessary. 
Finally we consider the divergence-free case. 
Recall that 
a $C^1$ vector field $w$ on 
a Riemannian manifold with a volume form $dvol$ is 
divergence-free 
if 
$L_w dvol = (\mathrm{div}\,  w) dvol$, 
where 
$L_w$ is the Lie derivative along $w$.

\begin{corollary}\label{prop:005}
Suppose that $w$ is regular and  divergence-free and has 
singularities. 
The following are equivalent: 
\\
1) 
$w$ is pointwise recurrent. 
\\
2)
$w$ is pointwise almost periodic. 
\\
3) 
$w$ is $R$-closed. 
\\
4)
$w$ is pointwise periodic. 
\\
5) $w$ has neither saddles nor $\partial$-saddles. 
\\
6) 
$w$ 
has neither saddles nor $\partial$-saddles 
and 
is structurally stable in the set of divergence-free vector fields 
up to topological equivalence. 
\\
7) 
$M/w$ is a closed interval. 
\\
If an equivalent condition holds, then 
$M$ is either a closed disk or a sphere and has one or two singularities, which are centers. 
\end{corollary}

\begin{proof}
By proposition \ref{prop:0055}, 
we obtain 
1) $\Leftrightarrow$ 
2) $\Leftrightarrow$ 
3)  $\Leftrightarrow$ 
4)  $\Leftrightarrow$ 
7). 
By taking the doubling of $M$ if necessary, 
we may assume  that 
$M$ has no boundaries. 
Obviously, 
4) $\Rightarrow$ 5).  
We show that 
5) $\Rightarrow$ 6).  
Suppose 5) holds. 
Since $w$ is regular and has singularities, 
by Poincar\'e-Hopf theorem,  
we have 
$M = S^2$.  
By Theorem 2.1.2\cite{MW}, 
$w$ is structurally stable in 
the set of divergence-free vector fields. 
Finally, 
we show that 6) $\Rightarrow$ 7). 
Suppose 6) holds. 
By Theorem 1.4.6.\cite{MW}, 
$w$ is determined by 
the saddle connection diagram (the union of saddle connection) 
up to topological equivalence. 
Since $w$ has no saddles,   
we obtain that  
$w$ has two centers and periodic orbits.  
By proposition \ref{prop:0055}, 
we have that 7) holds. 
\end{proof}

\section{Examples}

Recall a well-known fact that  
a point in 
a compact Hausdorff space is 
pointwise almost periodic 
if and only if 
its orbit closure is a minimal set. 
On the other hand, 
this is not true for non-compact cases.  
For instance, 
in Theorem 2.2 \cite{FK}, 
they have constructed 
a homeomorphism $f_{\mathcal{B}}$ on a non-compact 
Hausdorff space which is not $R$-closed 
but pointwise almost periodic as a flow 
such that there is an orbit closure which is not minimal. 
Note that 
a minimal set of an $R$-closed homeomorphism on 
a locally compact Hausdorff space needs 
not be compact (e.g. any non-trivial translation on $\Z$). 
%
Though 
there is 
a minimal mapping on a locally compact $T_1$ space 
which is pointwise almost periodic (see Theorem 3.2 \cite{FK}), 
%
there is no minimal homeomorphism on 
a locally compact non-compact space 
which is pointwise almost periodic as a flow. 
Otherwise 
fix a compact \nbd $C$ of a point $x$ 
and an open \nbd $U \subseteq C$ of $x$. 
The almost periodicity implies that 
$\bigcup_{i = 0}^N f^i(U) \supseteq O_f(x)$ 
for some $N > 0$. 
Since $X - \bigcup_{i = 0}^N f^i(U)$ is closed saturated, 
the minimality implies that 
$X - \bigcup_{i = 0}^N f^i(U) = \emptyset$. 
Therefore $X = \bigcup_{i = 0}^N f^i(C)$ is compact. 
Moreover the well-known fact is not true even for compact cases (see  following two examples). 
%
 
\begin{example}
Let $G$ be a trivial group and 
$X = \{ 0,1 \}$ a two point set with a topology $\{ \emptyset, X, \{1\} \}$. 
Since the closed sets are $ \emptyset, X, \{0\}$, 
the closed sets of the product space are 
$ \emptyset, \{(0, 0)\}, \{0\} \times X , X \times \{0\}, X \times X - \{(1,1)\}, X \times X$. 
Trivially $X$ is compact and 
$G$ is periodic. 
In particular,  
$G$ is pointwise almost periodic as a flow. 
$R = \{ (0,0) , (1,1), (1,0) \} = X \times X - \{(0,1)\}$ is not closed. 
Since $\overline{G 1} = X$, 
the $X/\widetilde{G} = X$ is not Hausdorff.  
Moreover the orbit of $1$ is not minimal but periodic. 
\end{example}

%
A next example shows 
them 
even if $X$ is  a $T_1$-space 

\begin{example}\label{ex:002}
Let $\alpha \in \mathbb{R} - \mathbb{Q}$ be an irrational number,   
$G = \mathbb{Z}$ a discrete group and 
$X = \mathbb{R}/ \mathbb{Z}$ a topological space with a topology 
$ \tau = \{ \emptyset  \} \sqcup \{ U, U \setminus F \mid U \text{ is cofinite in } X \}$, 
where $F := \mathbb{Z} \alpha/ \mathbb{Z} \subset X$. 
$G$ acts by $n \cdot [x] := [\alpha n + x ]$. 
Then 
$X$ is a compact $T_1$-space, 
$G$ is continuous and pointwise almost periodic  as a flow, 
$R$ is not closed, 
and 
$\overline{Gx}$ for any $x \in X - F$ is not minimal. 
\\
Indeed, 
notice that 
the set $C_X$ of closed subsets is  
$ \{ \emptyset, X  \} \sqcup \{ E, E \cup F \mid E \text{ is finite in } X \}$. 
For any $n \in G$, 
since $ n \cdot F = F$, 
we have that 
$ n \cdot C_X = C_X$ and so $G$ acts continuously.  
For any $x \in F$, each neighborhood  of $x$ is cofinite. 
Hence $X$ is compact. 
For any $x \in X$, 
every \nbd of $x$ contains all but finitely many points of its orbit
and so $G$ is pointwise almost periodic  as a flow. 
Since $\tau$ contains the cofinite topology, 
each point of $X$ is closed and so $X$ is $T_1$.  
For any $x \in X - F$, 
since  $\overline{Gx} = X$ and $F$ is a proper closed subset, 
we obtain that 
$\overline{Gx} = M$ is not minimal.  
For any $x, y  \in X - F$, 
we have 
$ Gx \times \{ y \} \subset R$.
Since any neighborhood  of $0$ is cofinite and so meets $Gx$, 
we obtain 
$0 \in \overline{Gx}$ and so $( 0, y ) \in \overline{Gx \times \{y \}} 
\subseteq \overline{R}$.    
Since $y \notin F = \overline{G 0}$,  
we have  
$( 0, y ) \notin R$.   
Thus $R$ is not closed. 
\end{example}
 
Replacing the topology in the above example, 
we obtain two homeomorphisms on non-compact Hausdorff spaces 
such that 
the first homeomorphism constructed in Proposition 1.5 \cite{ML}
is pointwise almost periodic as a flow 
and have a non-minimal orbit closure, 
and 
the second homeomorphism constructed in Proposition 1.7 \cite{ML}
has an orbit closure of an almost periodic point 
which contains a non almost periodic point. 
A next example also shows that Lemma \ref{lem:003} is not true 
 in general.

\begin{example}\label{ex:0021}
Let $\alpha \in \mathbb{R} - \mathbb{Q}$ be an irrational number,   
$X = \mathbb{R}/ \mathbb{Z}$, 
$F = \{ \alpha  n \mid n \in \mathbb{Z} \}/ \mathbb{Z}$, 
$\tau$ the topology on $X$ induced by the Euclidean topology on $\R$,  
and 
$E := 1_X \cup (F \times F)$ an equivalence relation on $X$.  
Equip with $X$ 
a topology generated by $\tau$ and 
$ \{ U \setminus F \mid U \in \tau \}$. 
Put $\F = \F_E$. 
Then 
$X$ is a non-compact  Hausdorff space, 
$\F$ is pointwise almost periodic, 
$R = \Eh$ is closed, 
and 
$X/\Ft$ is non-Hausdorff. 
Indeed, 
since $F$ is closed, 
we have 
$\Eh = E$ and so 
$\F$ is pointwise almost periodic. 
Now we show that $R$ is closed. 
Fix $x \neq y \in X$.  
If $x, y \notin F$, then
there are 
disjoint \nbds $U_x, U_y \in \tau$ of $x$ and $y$ on $X$. 
Then $(U_x \setminus F) \times (U_y \setminus F)$ 
is a \nbd of $(x, y)$ which does not meet $\Eh$. 
If $x \in F$ and $y \notin F$, 
then  
there are 
disjoint \nbds $U_x, U_y \in \tau$ of $x$ and $y$ on $X$. 
Then $U_x \times (U_y \setminus F)$ 
is a \nbd of $(x, y)$ which does not meet $\Eh$. 
This shows that $\Eh$ is closed. 
For $x \in X - F$, 
since $x$ and $F$ can't be separated by disjoint open sets, 
$X/\Ft$ is non-Hausdorff. 
\end{example}

A next example also shows that 
there is a flow on a compact surface such that 
$G$ is pointwise almost periodic 
and 
$R$ is not closed  but symmetric.

\begin{example}\label{ex:04}
Let $G $ be an additive group $\mathbb{R}$ (resp. $\mathbb{Z}$), 
$X = \{ (x, y) \mid x^2 + y^2 \leq 1 \}$ a 
unit closed disk, 
and  
$v = ( y ( 1- ( x^2 + y^2)), -x ( 1- ( x^2 + y^2)))$ a vector field.  
$G$ acts by $t \cdot (x, y) := v_t(x, y)$. 
Then the fix point set of $G$ is the union of the boundary and the origin, 
the other orbits are periodic (resp. almost periodic) orbits, 
and 
the orbit (resp. orbit class) space is not $T_1$. 
Hence  
$G$ is pointwise almost periodic  
and  $R$ is not closed but symmetric. 
\\
Indeed, 
the boundary points are singular but not separated by saturated neighborhoods. 
Let $S_r := \{ (x, y) \mid x^2 + y^2  = r^2 \}$ be a circle. 
Then $G$ acts $S_r$ as rotations. 
Hence $G$ is pointwise almost periodic. 
For $r_n :=1 - 1/ \sqrt{\pi n}$, $\mathbb{Z} \subseteq G$ acts $S_{r_n}$ as irrational rotations. 
Since $(r_n,  -r_n) \in R$ but $(1, -1) \notin R$, 
we have $R$ is a non-closed equivalence relation.  
\end{example}

The following example shows the existence of 
a pointwise periodic $R$-closed homeomorphism on 
a compact metrizable space which 
is not periodic.  

\begin{example}\label{ex:06}
Let 
$Y = S^1 \times 
\{ 1/n \in \mathbb{R} \mid n \in \mathbb{Z}_{>1} \}$ and 
$f: Y \to Y$ a homeomorphism by $f(x, y) := (x + y, y)$. 
Denote by $X$ the one point compatification of $Y$. 
Then canonical extension of $f$ with the new fixed point 
is also a homeomorphism.  
Hence  
$f$ is not periodic but pointwise periodic  
and 
$X$ is a compact metrizable space. 
Since the orbit space is Hausdorff, 
we have that 
$f$ is $R$-closed. 
\end{example}


\end{document}